\documentclass[letterpaper]{amsart}

\usepackage{amssymb}
\usepackage{amsthm}
\usepackage{amsmath}
\usepackage{enumitem}

\usepackage{euscript}
\usepackage{mathrsfs}

\usepackage{cancel}
\usepackage{soul}

\usepackage[normalem]{ulem}

\usepackage[foot]{amsaddr}

\usepackage[usenames]{xcolor}

\usepackage{hyperref}

\usepackage{tikz-cd}

\theoremstyle{plain}
\newtheorem{proposition}{Proposition}[section] 
\newtheorem{theorem}[proposition]{Theorem}
\newtheorem{lemma}[proposition]{Lemma}  
\newtheorem{corollary}[proposition]{Corollary}
\theoremstyle{definition}

\newtheorem{conjecture}[proposition]{Conjecture}

\theoremstyle{remark}
\newtheorem{remark}[proposition]{Remark}

\DeclareMathOperator{\Isom}{\mathsf{Isom}}

\DeclareMathOperator{\Vol}{Vol}

\DeclareMathOperator{\supp}{supp}

\DeclareMathOperator{\id}{id}
\DeclareMathOperator{\rank}{rank}

\DeclareMathOperator{\dist}{d}

\DeclareMathOperator{\Cc}{\mathcal{C}}
\DeclareMathOperator{\Fc}{\mathcal{F}}

\DeclareMathOperator{\Hc}{\mathcal{H}}

\DeclareMathOperator{\Oc}{\mathcal{O}}

\DeclareMathOperator{\Hb}{\mathbb{H}}

\DeclareMathOperator{\Nb}{\mathbb{N}}

\DeclareMathOperator{\Rb}{\mathbb{R}}

\DeclareMathOperator{\Gsf}{\mathsf{G}}
\DeclareMathOperator{\Psf}{\mathsf{P}}

\DeclareMathOperator{\SL}{\mathsf{SL}}

\newcommand{\abs}[1]{\left|#1\right|}
\newcommand{\norm}[1]{\left\|#1\right\|}

\newcommand{\ip}[1]{\left\langle #1\right\rangle}

\newcommand{\Ga}{\Gamma}
\newcommand{\ga}{\gamma}
\newcommand{\F}{\mathcal{F}}
\newcommand{\La}{\Lambda}

\newcommand{\Msf}{\mathsf{M}}
\newcommand{\Asf}{\mathsf{A}}

\newcommand{\Nsf}{\mathsf{N}}
\newcommand{\Ksf}{\mathsf{K}}

\newcommand{\Usf}{\mathsf{U}}

\newcommand{\fa}{\mathfrak{a}}

\newcommand{\mfg}{\mathfrak{g}}
\newcommand{\mfa}{\mathfrak{a}}

\newcommand{\R}{\mathbb{R}}

\newcommand{\ba}{\backslash}

\newcommand{\Leb}{\operatorname{Leb}}

\newcommand{\BMS}{\operatorname{BMS}}

\begin{document}

\title[The singularity conjecture and infinite covolume]{{A note on the singularity conjecture for infinite covolume discrete subgroups} }

\author[Kim]{Dongryul M. Kim}
\email{dongryul.kim@yale.edu}
\address{Department of Mathematics, Yale University, USA}

\author[Zimmer]{Andrew Zimmer}
\email{amzimmer2@wisc.edu}
\address{Department of Mathematics, University of Wisconsin-Madison, USA}

\date{\today}

 \subjclass[2020]{}

\begin{abstract} We consider  random walks on semisimple Lie groups where the support of the step distribution generates (as a group) a Zariski dense discrete subgroup of infinite covolume. When the semisimple Lie group has property (T), we show that the stationary measure on the Furstenberg boundary is singular to the Lebesgue measure class. This result does not require any condition on the moment or symmetry of the step distribution. When the semisimple Lie group has rank one and the step distribution has a finite first moment, we again show that the stationary measure on the Furstenberg boundary is singular to the Lebesgue measure class. For general semisimple Lie groups, we also obtain a sufficient condition  for the singularity of the stationary measure and a general Patterson--Sullivan measure.
\end{abstract}

\maketitle

\vspace{-1em}

\section{Introduction}

Let $\Gsf$ be a connected semisimple Lie group without compact factors and with finite center. Suppose $\Gamma < \Gsf$ is a Zariski dense discrete subgroup, and $\mathsf{m}$ is a probability measure on $\Gamma$ whose support generates $\Gamma$ (as a group). 
Let $\Fc$ denote the Furstenberg boundary, i.e. the flag manifold $\Fc := \Gsf / \Psf$ for a minimal parabolic subgroup $\Psf < \Gsf$. Then there is a unique \emph{$\mathsf{m}$-stationary measure} $\nu$ on $\Fc$ and is equal to the hitting measure for the random walk generated by $\mathsf{m}$ \cite{Furstenberg_boundary, GR_Furstenberg, GolMar1989}. The measure $\nu$ is also referred to as the \emph{Furstenberg measure}. 

The following is a well-known conjecture (cf. Kaimanovich--Le Prince \cite{KP_matrix}).
 
\begin{conjecture}[Singularity conjecture] \label{conj.singularity Liegp} If $\mathsf{m}$ has finite support, then the $\mathsf{m}$-stationary measure $\nu$ is singular to the Lebesgue measure class on $\Fc$. 
\end{conjecture} 

The Lie group $\Gsf$ has Kazhdan's property (T)  if and only if no simple factor of $\Gsf$ is locally isomorphic to $\mathsf{SO}(n,1)$ or $\mathsf{PU}(n,1)$ for any $n \geq 1$.
For such $\Gsf$ and $\Ga < \Gsf$ with infinite covolume, we give an affirmative answer to the singularity conjecture. We emphasize that in the following, $\mathsf{m}$ does not have \emph{any moment condition} and \emph{$\Ga$ does not need to be generated by the support of $\mathsf{m}$ as a semigroup}. 
 
\begin{theorem} \label{thm:singularity discrete group}  
Assume $\Gsf$ has property \emph{(T)}. If $\mathsf{m}$ is a probability measure on $\Gsf$ whose support generates (as a group) a Zariski dense discrete subgroup $\Gamma < \Gsf$ with 
$$
\Vol(\Ga \ba \Gsf) = + \infty,
$$
  then the $\mathsf{m}$-stationary measure  $\nu$ is singular to the Lebesgue measure class on $\Fc$. 
\end{theorem}   

We will show the same results for rank one Lie groups without property (T) in Theorem \ref{thm:rank one Leb} and Theorem \ref{thm:rank one case}, assuming the finite moment condition.

\begin{remark} \
\begin{enumerate}
\item In the special case when $\mathsf{m}$ is symmetric (i.e. $\mathsf{m}(g) = \mathsf{m}(g^{-1})$ for all $g \in \Gamma$) and the support of $\mathsf{m}$ generates $\Gamma$ as a semigroup, Lee--Tiozzo--Van Limbeek~\cite{LTvW2025} have independently obtained a proof of Theorem~\ref{thm:singularity discrete group}.
\item When $\Ga < \Gsf$ is a lattice, there exists $\mathsf{m}$ with infinite support such that the $\mathsf{m}$-stationary measure $\nu$ on $\Fc$ is in the Lebesgue measure class on $\Fc$, as shown by Furstenberg \cite{Furstenberg_RW}, Lyons--Sullivan \cite{LS_RW}, and Ballmann--Ledrappier \cite{BL_harmonic} (see also \cite{BQ_stationary}).

\item There exist examples where $\mathsf{m}$ is a finitely supported probability measure on $\Gsf$, the group generated by the support of $\mathsf{m}$ is non-discrete, and the $\mathsf{m}$-stationary measure is absolutely continuous with respect to the Lebesgue measure class on $\Fc$. For $\Gsf = \SL(2, \Rb)$, this was shown by B\'ar\'any--Pollicott--Simon \cite[Theorem 26, Section 8]{BPS_stationary}, and Bourgain constructed such a measure $\mathsf{m}$ with  symmetric support \cite{Bourgain_abscont}. For $\Gsf \neq \SL(2,\Rb)$, Benoist--Quint gave a construction of such $\mathsf{m}$ with symmetric support \cite[Theorem 1.2]{BQ_stationary}. 
\item In contrast to the rigidity in Theorem~\ref{thm:singularity discrete group}, Dey--Hurtado~\cite{DeyHurtado} have constructed examples of non-lattice discrete subgroups $\Gamma  < \Gsf$ which act minimally on $\Fc$ and in particular there exist stationary measures coming from random walks on such groups which have full support in $\Fc$.
\end{enumerate}
\end{remark} 

The proof of Theorem~\ref{thm:singularity discrete group} is based on studying the growth indicator function of $\Gamma$ (defined in the next subsection). When $\Gsf$ has property (T) and $\Gamma$ does not have a lattice factor, one can use estimates on matrix coefficients of unitary representations to prove a gap result for this function: this is a well-known theorem of Quint \cite{Quint_Kazhdan} for $\Gsf$ higher rank simple (also \cite{LO_Dichotomy} for a shorter  proof) and of Corlette \cite{Corlette_gap} for $\Gsf$ rank one. In Proposition~\ref{prop:gap result}, we observe that this gap result extends to the semisimple case.
On the other hand, using a result of B\'enard \cite{Benard_RW} about random walks and work of Lee--Oh \cite{LO_Dichotomy} on a relation between Lebesgue measure and the growth indicator function, we will show that the non-singularity of the stationary measure and Lebesgue measure class implies that the growth indicator function violates this gap.

\subsection{Patterson--Sullivan measures and the growth indicator function} Our argument will exploit the fact that the Lebesgue measure class on $\Fc$ contains a ``higher rank Patterson--Sullivan measure.'' In this subsection we recall the definition of these measures and the definition of the growth indicator function. Then we explain how non-singularity between a stationary measure and a Patterson--Sullivan measure restricts the growth indicator function. This result does not require that $\Gsf$ has property (T).

Fix a maximal real split torus $\Asf < \Psf$ and write $\mfa := \log \Asf$.  
For $\Ga < \Gsf$, a functional $\phi : \mfa \to \Rb$, and $\delta \ge 0$, a Borel probability measure $\mu$ on $\F$ is called a \emph{$\delta$-dimensional coarse $\phi$-Patterson--Sullivan measure} of $\Ga$ if there exists $C \ge 1$ such that for any $g \in \Ga$, the measures $\mu$, $g_*\mu$ are absolutely continuous and
\begin{equation}\label{eqn:PS intro}
\frac{1}{C} e^{- \delta \phi(B(g^{-1}, x))} \le \frac{d g_* \mu}{d \mu}(x) \le C e^{- \delta \phi(B(g^{-1}, x))} \quad \text{for }\mu\text{-a.e.} \quad x \in \Fc,
\end{equation}
where $B : \Gsf \times \F \to \fa$ is the Iwasawa cocycle (see Section~\ref{subsec:PS}).
When $C = 1$ and hence equality holds above, we call $\mu$ a \emph{$\delta$-dimensional $\phi$-Patterson--Sullivan measure}. These higher rank Patterson--Sullivan measures were introduced and constructed in \cite{Quint_PS}. We also note that in the literature these measures are sometimes called \emph{(quasi-)conformal measures}.

We choose a positive closed Weyl chamber $\mfa^+ \subset \mfa$ and denote the associated Cartan projection by $\kappa : \Gsf \to \mfa^+$ (see Equation~\eqref{eqn:Cartan projection}). Given a discrete subgroup $\Ga < \Gsf$,  the \emph{growth indicator function} $\psi_{\Ga} : \mfa \to \R_{\ge 0} \cup \{-\infty\}$ is defined by 
$$
\psi_{\Ga}(u) := \norm{u} \cdot \inf_{\Cc \ni u} \left\{ \text{critical exponent of } s \mapsto \sum_{\ga \in \Ga, \kappa(\ga) \in \Cc} e^{-s \norm{\kappa(\ga)}} \right\}
$$ 
where $\norm{\cdot}$ is any norm  on $\mfa$ and the infimum is over all open cones $\Cc \subset \mfa$ containing $u$. Notice that $\psi_{\Ga} \equiv -\infty$ on $\mfa \smallsetminus \mfa^+$. The growth indicator function can be viewed as a higher rank analogue of the critical exponent in rank one settings and was introduced by Quint \cite{Quint_Divergence}.

We will show that non-singularity between a stationary measure and a Patterson--Sullivan measure restricts the growth indicator function.

\begin{theorem} \label{thm.general PS}
   Suppose $\mathsf{m}$ is a probability measure on $\Gsf$ whose support generates (as a group) a Zariski dense discrete subgroup $\Ga < \Gsf$ and $\mu$ is a $\delta$-dimensional coarse $\phi$-Patterson--Sullivan measure of $\Ga$. If the $\mathsf{m}$-stationary measure $\nu$ on $\Fc$ is non-singular to $\mu$, then
   $$
   \sum_{\ga \in \Ga} e^{-\delta \phi(\kappa(\ga))} = + \infty \quad \text{and} \quad \psi_{\Ga} \le \delta \cdot \phi \quad \text{on} \quad \mfa
   $$
   with equality at some $u \in \mfa^+ \smallsetminus \{0\}$.

\end{theorem}

When $\mathsf{m}$ has finite superexponential moment (e.g. finite support), $\Ga$ is relatively hyperbolic as an abstract group, and $\Gamma$ is generated by $\supp \mathsf{m}$ as a semigroup,  we proved the stronger statement \cite[Theorem 1.15]{KZ_Rigidity}: if the $\mathsf{m}$-stationary measure is non-singular to a $\delta$-dimensional coarse  $\phi$-Patterson--Sullivan measure of $\Ga$, then \begin{equation} \label{eqn:KZ_Green}
   \sup_{\ga \in \Ga} | \dist_{G}(\id, \ga) - \delta \phi(\kappa(\ga))| < + \infty
\end{equation}
 where $\dist_G$ denotes the Green metric for $\mathsf{m}$ on $\Ga$. In particular, any orbit map of $\Ga$ into the symmetric space associated to $\Gsf$ is a quasi-isometric embedding (with respect to a word metric on $\Ga$). It would be interesting to know if this is true for more general $\Ga$ and $\mathsf{m}$. 

\subsection{The rank one case}  We now consider the case when  $\rank \Gsf = 1$ and denote by $(X, \dist_X)$ the associated symmetric space of $\Gsf$. 
Then $X$ is a negatively curved symmetric space and when $X$ is a real or complex hyperbolic space, $\Gsf$ does not have  property (T). 

We fix a basepoint $o \in X$. In this setting, the Furstenberg boundary $\Fc$ is identified with the visual boundary $\partial_{\infty} X$, and for a discrete subgroup $\Ga < \Gsf$, its limit set $\La(\Ga) \subset \partial_{\infty} X$  is the set of all accumulation points of $\Ga (o) \subset X$. The group $\Ga$ is \emph{non-elementary} if $\# \La(\Ga) \ge 3$.

Given a probability measure $\mathsf{m}$ on $\Gsf$, if the support of $\mathsf{m}$ generates a non-elementary discrete subgroup as a group, then there exists a unique $\mathsf{m}$-stationary measure on $\partial_{\infty} X$ which is equal to the hitting measure for the random walk \cite[Remark following Theorem 7.7]{Kaimanovich2000}.
The measure $\mathsf{m}$ has a \emph{finite first moment} if
$$
\int_{\Gsf} \dist_{X}(o, g o) d \mathsf{m}(g) < + \infty.
$$
Under the finite first moment condition, we prove the singularity conjecture (Conjecture \ref{conj.singularity Liegp}) for non-lattice discrete subgroups of $\Gsf$.

\begin{theorem} \label{thm:rank one Leb}
Suppose $\rank \Gsf = 1$. If $\mathsf{m}$ is a probability measure on $\Gsf$ whose support generates (as a group) a non-elementary discrete subgroup $\Ga < \Gsf$ with $$
\Vol(\Ga \ba \Gsf) = + \infty
$$
and $\mathsf{m}$ has a finite first moment, then the $\mathsf{m}$-stationary measure $\nu$ is singular to the Lebesgue measure class on $\partial_{\infty} X$.
\end{theorem}

\begin{remark} 
   
   In the special case that
   \begin{itemize}
   \item $X = \Hb^n$ is a real hyperbolic space,
   \item  $\Ga < \Gsf$ is a non-uniform lattice (i.e., $\Ga \ba X$ is non-compact and has finite volume), and
   \item the support of $\mathsf{m}$ generates $\Ga$ as a \textbf{semigroup},
   \end{itemize}
   the singularity of $\mathsf{m}$-stationary measure on $\partial_{\infty} X$ and the Lebesgue measure class were known under various moment conditions with respect to a \textbf{word metric} on $\Ga$.

   When $n = 2$, Gadre--Maher--Tiozzo~\cite{GMT_word} showed the singularity for $\mathsf{m}$ with finite first moment with respect to a word metric on $\Ga$. See also \cite{GL_geodesic, DKN_circle, KP_matrix} for different contexts. For $n \ge 3$, this was obtained by Randecker--Tiozzo~\cite{RT_cusp}, under the moment condition that $\mathsf{m}$ has a finite $k$-th moment with respect to a word metric for some $k > n-1$, and a finite exponential moment for a so-called relative metric on $\Ga$.
   
   Further, we remark that Kosenko--Tiozzo \cite{KT_cocompact} explicitly constructed cocompact lattices $\Ga$ of $\Isom(\Hb^2)$ and finitely supported probability measures $\mathsf{m}$ such that the $\mathsf{m}$-stationary measure on $\partial_{\infty} \Hb^2$ is singular to the Lebesgue measure class on $\partial_{\infty} \Hb^2$.

\end{remark}

Theorem \ref{thm:rank one Leb} is a consequence of a more general result about Patterson--Sullivan measures.

Recall that the Busemann cocycle $\beta: \Gsf \times \partial_\infty X \rightarrow \Rb$ is defined by 
$$
\beta(g, x) := \lim_{p \to x} \dist_X(p, g^{-1} o) - \dist_X(p, o).
$$
Then, for $\Ga < \Gsf$ and $\delta \ge 0$, a Borel probability measure $\mu$ on $\partial_\infty X$ is called a \emph{$\delta$-dimensional Patterson--Sullivan measure} of $\Ga$ if for any $g \in \Ga$, the measures $\mu$, $g_*\mu$ are absolutely continuous and
\begin{equation}\label{eqn:PS hyperbolic} 
\frac{d g_* \mu}{d \mu}(x) = e^{- \delta \beta(g^{-1}, x)} \quad \text{for }\mu\text{-a.e.} \quad x \in \partial_{\infty} X.
\end{equation}
In fact these are a special case of the measures introduced in Equation~\eqref{eqn:PS intro}. In this rank one setting, we can identify $\fa$ with $\Rb$ and $\mfa^+$ with  $\Rb_{\ge 0}$ in such a way that the Cartan projection and the Iwasawa cocycle are the displacement of the basepoint $o \in X$ and the Busemann cocycle respectively:  for $g \in \Gsf$ and $x \in \partial_{\infty} X$,
$$\kappa(g) = \dist_X(o, g o) \quad \text{and} \quad B(g, x) = \beta(g,x).$$
Then Equation~\eqref{eqn:PS hyperbolic}  is exactly Equation~\eqref{eqn:PS intro}  with $C = 1$ and the functional $\phi$ which coincides with our identification $\mfa = \Rb$. 

For a non-elementary discrete subgroup $\Ga < \Gsf$ and a Patterson--Sullivan measure $\mu$ on $\partial_{\infty} X$ of $\Ga$, we denote by $m_{\mu}^{\BMS}$ the (generalized) Bowen--Margulis--Sullivan measure of $\Ga$ on the unit tangent bundle $\mathrm{T}^{1}(\Ga \ba X)$ associated to $\mu$ (see Section~\ref{subsec:BMS}).

\begin{theorem} \label{thm:rank one case}
Suppose $\rank \Gsf = 1$, $\Ga < \Gsf$ is a non-elementary discrete subgroup, and $\mu$ is a Patterson--Sullivan measure of $\Ga$ such that $$
m_{\mu}^{\BMS}(\mathrm{T}^{1}(\Ga \ba X)) = + \infty.
$$
If $\mathsf{m}$ is a probability measure on $\Ga$  whose support generates $\Ga$ (as a group) and $\mathsf{m}$ has a finite first moment,
then the $\mathsf{m}$-stationary measure $\nu$ on $\partial_{\infty} X$ is singular~to~$\mu$.
\end{theorem}

We can always take $\mu$ in the Lebesgue measure class on $\partial_{\infty} X$, and in this case $m_{\mu}^{\BMS}(\mathrm{T}^{1}(\Ga \ba X)) = + \infty$ if and only if $\Vol(\Ga \ba G) = + \infty$. Hence, Theorem \ref{thm:rank one Leb} follows from Theorem \ref{thm:rank one case}.

We remark that both Theorem \ref{thm.general PS} and Equation~\eqref{eqn:KZ_Green} also apply to rank one cases, regardless of the size of Bowen--Margulis--Sullivan measures.

\begin{remark} \label{rmk:general PS}
In case that
   \begin{itemize}
   \item $X$ is a proper geodesic Gromov hyperbolic space,
   \item  $\Ga < \Isom(X)$ is a finitely generated but not quasi-convex discrete subgroup,
   \item the support of $\mathsf{m}$ generates $\Ga$ as a \textbf{semigroup}, and
   \item $\mathsf{m}$ has a finite superexponential moment with respect to a \textbf{word metric} on~$\Ga$,
   \end{itemize}
   the singularity of the $\mathsf{m}$-stationary measure and Patterson--Sullivan measures of $\Ga$ on $\partial_{\infty} X$ were studied by several authors.

   When $\Ga$ is word hyperbolic, $X$ admits a geometric action, and $\mathsf{m}$ further has a symmetric and finite support, Blach{\`e}re--Ha\"issinsky--Mathieu~\cite{BHM2011} showed such a singularity. 
   When $\Ga$ acts geometrically finitely on $X$, this singularity is due to Gekhtman--Tiozzo~\cite{GT2020}.
   More generally, when $\Ga$ is relatively hyperbolic as an abstract group (e.g. any finitely generated Kleinian group), this was obtained in our earlier paper \cite{KZ_Rigidity}.

\end{remark}

\subsection*{Acknowledgements} 

We thank Hee Oh for helpful discussions. We also thank Timoth\'ee B\'enard and Ilya Gekhtman for suggesting the formulation of Theorems \ref{thm:rank one Leb} and \ref{thm:rank one case} and for helpful discussions on its proof. Kim extends his gratitude to his advisor Hee Oh for her encouragement and guidance. 

Kim thanks the University of Wisconsin--Madison for hospitality during the \emph{Midwest Summer School in Geometry, Topology, and Dynamics} in June 2025 (which was supported by grant DMS-2230900 from the National Science Foundation), where work on this project was finished. Zimmer was partially supported by a Sloan research fellowship and grant DMS-2105580 from the National Science Foundation.

\section{Discrete subgroups of Lie groups} \label{sec.Liegps}

We fix a Langlands decomposition $\Psf = \Msf \Asf \Nsf$ where $\Asf$ is a maximal real split torus of $\Gsf$, $\Msf$ is the maximal
compact subgroup of $\Psf$ commuting with $\Asf$, and $\Nsf$ is the unipotent radical of $\Psf$. The Furstenberg boundary is defined as
$$
\Fc := \Gsf / \Psf. 
$$

Let $\mfg$ and $\mfa$ be the Lie algebras of $\Gsf$ and $\Asf$ respectively. 
We denote by $\fa^*$ the space of all functionals $\mfa \to \R$.  Fix a maximal compact subgroup $\Ksf < \Gsf$ and a positive closed Weyl chamber $\mfa^+ \subset \mfa$ with $\Asf^+ := \exp \fa^+$ so that we have a Cartan decomposition $\Gsf = \Ksf \Asf^+ \Ksf$. Then let 
\begin{equation} \label{eqn:Cartan projection}
   \kappa:\Gsf\to\mathfrak{a}^+
\end{equation}
denote the associated Cartan projection, i.e. $\kappa(g) \in \mathfrak{a}^+$ is the unique element where $g = k (\exp \kappa(g) ) \ell$ for some $k, \ell \in \Ksf$.

When $\Gsf = \prod_{i=1}^m \Gsf_i$ is a product of simple Lie groups, we assume that 
$$
\fa = \oplus_{i=1}^m \fa_i, \quad \fa^+ = \oplus_{i=1}^m \fa_i^+, \quad \Ksf = \prod_{i=1}^m \Ksf_i, \quad \text{and} \quad \Psf =  \prod_{i=1}^m \Psf_i
$$
where $\fa_i$, $\fa_i^+$, $\Ksf_i$, and $\Psf_i$ are the corresponding objects for $\Gsf_i$.

\subsection{Shadows and conical limit sets} \label{subsec.convergence theta bdr}

Let $X := \Gsf/\Ksf$ be the associated symmetric space and fix the basepoint $o := [\id] \in \Gsf / \Ksf$. Fix a $\Ksf$-invariant norm $\| \cdot \|$ on $\fa$ induced from the Killing form, and let $\dist_X$ denote the $\Gsf$-invariant symmetric Riemannian metric on $X$ defined by 
$$
\dist_X(g o, h o) = \| \kappa(g^{-1} h) \|
$$
 for $g, h \in \Gsf$.

For $p \in X$ and $R > 0$, let 
$$
B_X(p, R) := \{ z \in X : \dist_X(p, z) < R\}.
$$
Then, for $q \in X$, the \emph{shadow} of $B_X(p, R)$ viewed from $q$ is 
\begin{equation} \label{eqn:shadow def}
    \Oc_R(q, p)   := \{ g\Psf \in \Fc : g \in \Gsf, \ go = q, \ gA^+o \cap B_X(p, R) \neq \emptyset \}.
\end{equation}
Given a subgroup $\Ga < \Gsf$, we define the \emph{conical limit set} of $\Ga$ as follows:
$$
\La^{\rm con}(\Ga) := \left\{ x \in \Fc : \begin{matrix}
\exists \ R > 0 \text{ and an infinite sequence } \{\ga_n\} \subset \Ga \\
\text{s.t. } x \in \Oc_R(o, \ga_n o) \text{ for all } n \ge 1
\end{matrix}  \right\}.
$$

\subsection{Iwasawa cocycles and Patterson--Sullivan measures} \label{subsec:PS}
Using the identification $\Fc = \Ksf/\Msf$, the \emph{Iwasawa cocycle} $B : \Gsf \times \Fc \to \fa$ is defined as follows: for $g \in \Gsf$ and $x \in \Fc$, fix $k \in \Ksf$  such that $k \Msf = x$ and let $B(g, x) \in \fa$ be the unique element such that
$$
gk \in \Ksf \left( \exp B(g, x) \right) \Nsf.
$$
Then $B(\cdot, \cdot)$ satisfies the cocycle relation: for any $x \in \Fc$ and $g_1, g_2 \in \Gsf$,
$$
B(g_1 g_2, x) = B(g_1, g_2 x) + B(g_2, x).
$$
Recall that Patterson--Sullivan measures in higher rank were defined, using the Iwasawa cocycle, in Equation~\eqref{eqn:PS intro}.

Let $\Leb$ be the unique $\Ksf$-invariant probability measure on $\Fc$, which is induced by a smooth volume form and hence in the Lebesgue measure class. A crucial observation in our proof is that $\Leb$ is a Patterson--Sullivan measure.

\begin{lemma} \cite[Lemma 6.3]{Quint_PS} \label{lem:Leb PS}
   The measure $\Leb$ is a 1-dimensional $\Psi_{\Gsf}$-Patterson--Sullivan measure of $\Gsf$, where $$
\Psi_{\Gsf} \in \mfa^*
$$
is the sum of all positive roots for $(\mfg, \mfa^+)$.
\end{lemma}

Next we recall a version of the classical Shadow Lemma. The following was  proved in \cite{LO_Invariant}  for non-coarse Patterson--Sullivan measures, but the same proof works for coarse ones (see also \cite[Theorem 3.3]{Albuquerque_PS} for the Shadow Lemma for a smaller class of functionals).

\begin{lemma}[{Shadow Lemma, \cite[Lemma 7.8]{LO_Invariant}}] \label{lem.shadow lemma}
   Let $\Ga < \Gsf$ be a Zariski dense discrete subgroup and $\mu$ a $\delta$-dimensional coarse $\phi$-Patterson--Sullivan measure of $\Ga$ on $\Fc$, for $\phi \in \mfa^*$ and $\delta \ge 0$. Then for any $R > 0$ large enough, there exists $C = C(R) > 1$ such that
   $$
   \frac{1}{C} e^{- \delta \phi (\kappa(\ga))} \le \mu(\Oc_R(o, \ga o)) \le C e^{- \delta \phi  (\kappa(\ga))} \quad \text{for all} \quad \ga \in \Ga.
   $$
\end{lemma}

\begin{remark} There are a number of other versions of shadows and Shadow Lemmas in higher rank, see for instance  \cite{Quint_PS,MR4668098}. 
\end{remark}

\subsection{Random walks} \label{subsec:RW} In this section we recall some classical properties of random walks on semisimple Lie groups and a result of B\'enard. For more background, see \cite{Furman_survey, BQ_book}.

Let $\Ga < \Gsf$ be a Zariski dense subgroup. Let $\mathsf{m}$ be a probability measure on $\Gsf$ whose support generates $\Ga$ as a group, i.e.,
$$
 \langle \supp \mathsf{m} \rangle = \Ga.
$$
In this case, the semigroup generated by $\supp \mathsf{m}$ is Zariski dense as well \cite[Lemma 6.15]{BQ_book}.

Consider the random walk
$$\omega_n = \ga_1 \cdots \ga_n$$
 where the $\ga_i$'s are independent identically
distributed elements of $\Ga$ each with distribution $\mathsf{m}$.
For each $n \in \Nb$, we fix a Cartan decomposition
$$
\omega_n = k_n (\exp \kappa(\omega_n)) \ell_n \in \Ksf \Asf^+ \Ksf
$$
so that $k_n \in \Ksf$ and $\kappa(\omega_n) \in \fa^+$ are random variables. 

We will use the following well-known fact about the convergence of random products (for a proof, see for instance \cite[Lemmas 2.1, 2.8]{Benard_RW}). 

\begin{lemma}\label{lem:convergence} 
   For almost every sample path $\omega = (\omega_n) \in \Gsf^{\Nb}$, there exists $x_\omega \in \Fc$ such that as $n \to + \infty$,
   $$
   k_n \Psf \to x_\omega \quad \text{in} \quad \F \quad \text{and} \quad \alpha(\kappa(\omega_n)) \to + \infty \quad \text{for all simple roots} \quad \alpha \in \mfa^*.
   $$
\end{lemma}

Continuing the notation in the above lemma, B\'enard showed that the convergence is conical as follows (we rephrase this result in terms of shadows).

\begin{theorem} \cite[Theorem A bis]{Benard_RW} \label{thm:conical RW}
   For any $\epsilon > 0$, there exists $R > 0$ such that for almost every $\omega = (\omega_n) \in \Gsf^{\Nb}$, 
   $$
\liminf_{N \to + \infty} \frac{1}{N} \# \{ 1 \le n \le N : x_\omega \in \Oc_R(o, \omega_n o) \} > 1 - \epsilon.
   $$
\end{theorem}

By work of Guivarc'h--Raugi \cite{GR_Furstenberg} and Gol'dshe{\u i}d--Margulis~\cite{GolMar1989}, there exists a unique $\mathsf{m}$-stationary measure $\nu$ on $\Fc$. Further, for a Borel subset $E \subset \Fc$, 
\begin{equation}\label{eqn:stat meas as hitting meas}
\nu(E) = \operatorname{Prob} \left( \omega = (\omega_n) \in \Gsf^{\Nb} : x_{\omega} \in E \right)
\end{equation} 
(see also~\cite[Proposition 10.1]{BQ_book}).

The following is a consequence of Lemma~\ref{lem:convergence}, Theorem~\ref{thm:conical RW}, and Equation~\eqref{eqn:stat meas as hitting meas}.

\begin{corollary}\label{cor:conical has positive measure}
    With the notation above,
   $$
   \nu(\La^{\rm con}(\Ga)) = 1.
   $$
\end{corollary}

\section{Gap theorem for the growth indicator function}

In proving Theorem \ref{thm:singularity discrete group}, we will employ the following result of Lee--Oh. Recall the functional $\Psi_{\Gsf}$ defined in Section~\ref{subsec:PS}.

\begin{proposition} \cite[Proposition 7.6]{LO_Dichotomy} \label{prop:LO76}
   Suppose $\Ga < \Gsf$ is a discrete subgroup such that $\psi_{\Ga} < \Psi_{\Gsf}$ on $\mfa^+ \smallsetminus \{0\}$. Then $$\Leb(\La^{\rm con}(\Ga)) = 0.$$
\end{proposition}

To apply Proposition \ref{prop:LO76}, we first prove the gap result for the growth indicator function as follows. While the following cannot be directly applied, we will proceed with a reduction argument to use it in the next section.

\begin{proposition}\label{prop:gap result} Suppose $\Gsf= \prod_{i =1}^m \Gsf_i$ is a product of  connected simple Lie groups of non-compact type with trivial centers and $\Gamma < \Gsf$ is a Zariski dense non-lattice discrete subgroup.   Assume
   \begin{enumerate} 
\item $\Gsf$ has property \emph{(T)}, and
\item for each $1 \leq i \leq m$ there is no non-zero $\Gsf_i$-invariant vector in $L^2(\Gamma \backslash \Gsf)$. 
\end{enumerate} 
Then 
$$
\psi_{\Ga} <  \Psi_{\Gsf} \quad \text{on} \quad \fa^+\smallsetminus\{0\}.
$$
\end{proposition} 

\begin{remark}
The above proposition was known when $\Gsf$ has rank one  \cite[Theorem 4.4]{Corlette_gap} and when $\Gsf$ is simple with higher rank (\cite{Quint_Kazhdan}, \cite[Theorem 7.1]{LO_Dichotomy}).  For semisimple $\Gsf$, it is not enough to assume that $\Ga$ is not a lattice. Indeed, when $L^2(\Gamma \backslash \Gsf)$ contains a  non-zero $\Gsf_i$-invariant vector for some $i$, it is possible for the conclusion of the proposition to fail. For instance if $\Gsf=\Gsf_1 \times \Gsf_2$ and $\Ga = \Ga_1 \times \Ga_2$ where $\Gamma_1 < \Gsf_1$ is a lattice and $\Ga_2 < \Gsf_2$ is a Zariski dense non-lattice discrete subgroup, then  
$$
\psi_{\Ga} =  \Psi_{\Gsf} \quad \text{on} \quad \fa_1^+ \oplus \{0\}.
$$
(Since $\fa^+$ is the closed maximal Weyl chamber, $\fa_1^+ \oplus \{0\} \subset \fa^+$.)
\end{remark}

\begin{remark} 
   One can see that the condition (2) in Proposition \ref{prop:gap result} holds when the subgroup $\Gamma \Gsf_i$ is dense in $\Gsf$ for each $1 \leq i \leq m$.
\end{remark}

The proof below closely follows the argument in ~\cite[Theorem 7.1]{LO_Dichotomy} and we use the following result from that paper to reduce to a problem about unitary representations. Let $\rho$ be the unitary representation of $\Gsf$ on $L^2( \Gamma \backslash \Gsf)$ given by $\rho(g) f(x) = f(xg)$ for $g \in \Gsf$, $f \in L^2(\Ga \ba \Gsf)$, and $x \in \Ga \ba \Gsf$.

\begin{proposition}\cite[Proposition 7.3]{LO_Dichotomy}\label{prop:LO result} 
   Suppose   
   $\theta \in \mfa^*$ satisfies
   \begin{enumerate} 
      \item $0 < \theta < \Psi_{\Gsf}$ on $\mfa^+ \smallsetminus \{0\}$, and 
      \item there exists a function $C : (0,+\infty) \rightarrow (0,+\infty)$ so that 
$$
\ip{ \rho(\exp u) f, f} \leq C(\epsilon) e^{-(1-\epsilon)\theta(u)}\norm{f}_2^2
$$
for any $\Ksf$-invariant vector $f \in L^2( \Gamma \backslash \Gsf)$, $u \in \fa^+$, and $\epsilon > 0$. 
   \end{enumerate}
Then  
$$
\psi_{\Ga} \leq \Psi_{\Gsf} - \theta \quad \text{on} \quad \fa^+ \smallsetminus \{ 0\}.
$$
 
\end{proposition} 

\begin{proof}[Proof of Proposition~\ref{prop:gap result}]
    We can decompose
$$
\rho = \int^{\oplus}_Y \rho_y dy
$$
into irreducible representations $\rho_y : \Gsf \rightarrow \Usf(\Hc_y)$, for some (standard) measure space $(Y, dy)$ (see \cite[Section 2.3]{Zimmer_book}).
As in Section 2.6 of~\cite{Oh2002}, we can write $\Hc_y= \otimes_i \mathcal{H}_{y,i}$ and $\rho_y = \otimes_i \rho_{y,i}$ where  each $\rho_{y,i} : \Gsf_i \rightarrow \mathsf{U}(\mathcal{H}_{y,i})$ is irreducible. Further, any $\Ksf$-invariant $f \in \Hc_y$ can be written as $f = \otimes_i f_i$ where each $f_i$ is $\Ksf_i$-invariant.

Now, if $f =( f_y)_y \in L^2( \Gamma \backslash \Gsf)$ is $\Ksf$-invariant, then $f_y$ is $\Ksf$-invariant for a.e. $y$ and for such $y$, $f_y= \otimes_i f_{y,i}$ where each $f_{y,i}$ is $\Ksf_i$-invariant. Further, 
$$
\ip{ ( \otimes_i f_{y,i})_y , ( \otimes_i g_{y,i})_y } = \int_Y \prod_i \ip{f_{y,i},g_{y,i}} dy.  
$$
So using Proposition~\ref{prop:LO result} it suffices to show the following. 

\medskip

\noindent \textbf{Claim:} For each $1 \leq i \leq m$ there exist $\theta_i \in \fa_i^*$ and a function $C_i : (0,+\infty) \rightarrow (0,+\infty)$ such that $0 < \theta_i < \Psi_{\Gsf}$ on $\fa^+_i \smallsetminus \{0\}$ and 
$$
\ip{ \rho_{y,i}(\exp u) f, f} \leq C_i(\epsilon) e^{-(1-\epsilon)\theta_i(u)}\norm{f}_2^2
$$
for a.e. $y \in Y$, any $\Ksf_i$-invariant vector $f \in \mathcal{H}_{y,i}$, $u \in \fa_i^+$, and $\epsilon > 0$. 

\medskip

Fix $1 \leq i \leq m$. By hypothesis and construction, $\rho_{y,i}$ has no non-zero $\Gsf_i$-invariant vectors for a.e. $y \in Y$. Hence when $\rank(\Gsf_i) > 1$, the claim follows immediately from Theorem 1.2 (and the estimates for $\xi_{\mathscr{S}}$ on page 136) in \cite{Oh2002}. 

Suppose $\rank(\Gsf_i) = 1$. Since $\Gsf$ has property (T), $\Gsf_i$ is not isomorphic to $\mathsf{SO}(n,1)$ or $\mathsf{PU}(n,1)$ for any $n \ge 1$.  Then by ~\cite[Theorem 2.5.3]{Cow1979} there exists $p > 0$ (which only depends on $\Gsf_i$) such that for a.e. $y \in Y$, each matrix coefficient  
$$
g \in \Gsf_i \mapsto \ip{\rho_{y,i}(g) f_1,f_2}
$$
is in $L^q(\Gsf_i)$ for all $q \geq p$. For such $y \in Y$, \cite[Corollary pg. 108]{CHH1988} implies that there exists $k \in \Nb$ (which only depends on $\Gsf_i$) such that 
$$
\ip{ \rho_{y,i}(g) f, f} \leq \Xi_i(g)^{1/k}\norm{f}_2^2
$$
for any $\Ksf_i$-invariant vector $f \in \mathcal{H}_{i,y}$, where $\Xi_i$ is the Harish-Chandra function of $\Gsf_i$. Then the claim follows from the definition of $\Xi_i$.
\end{proof}

\section{Proof of Theorem \ref{thm:singularity discrete group}}

We are now ready to prove Theorem~\ref{thm:singularity discrete group}. Let $\Gsf$ be a connected semisimple Lie group without compact factors and with finite center. Suppose $\Ga < \Gsf$ is a Zariski dense discrete subgroup, $\mathsf{m}$ a probability measure whose support generates $\Ga$ as a group, and $\nu$ is the unique $\mathsf{m}$-stationary measure on the Furstenberg boundary $\Fc$. 

Notice that replacing $\Gsf$ by $\Gsf/Z(\Gsf)$, $\Gamma$ by its projection to $\Gsf/Z(\Gsf)$, and $\mathsf{m}$ by its push-forward to $\Gsf/Z(\Gsf)$ does not change $\Fc$ or  $\nu$. So for our purposes there is no loss in generality in assuming that $Z(\Gsf) = \{\id\}$. Then we can write 
$$
\Gsf = \prod_{i=1}^m \Gsf_i
$$
where the $\Gsf_i$'s are the simple factors of $\Gsf$. 

We prove Theorem~\ref{thm:singularity discrete group} by contradiction. To that end, suppose that 
\begin{itemize}
\item $\nu$ is non-singular to the Lebesgue measure $\Leb$ on $\Fc$,
\item $\Gamma$ is not a lattice, and
\item $\Gsf$ has property (T).
\end{itemize} 

To use Proposition~\ref{prop:gap result} we need to reduce to the case where  condition (2) is satisfied. 

\begin{proposition} There exists $J \subset \{1,\dots, m\}$ non-empty such that
\begin{enumerate}
\item the projection $\Gamma'$ of $\Gamma$ to  $\Gsf':= \prod_{i \in J} \Gsf_i$ is a Zariski dense non-lattice discrete subgroup of $\Gsf'$, and 
\item for every $j \in J$, there is no $\Gsf_j$-invariant non-zero vector in $L^2(\Gamma' \backslash \Gsf')$. 
\end{enumerate} 
\end{proposition}

\begin{proof}
    It suffices to consider the case where there is a $\Gsf_1$-invariant non-zero vector $f \in L^2(\Gamma \backslash \Gsf)$. Replacing $f$ by $\min\{f,1\}$ we can assume that $f$ is bounded. Since $\Gamma$ is not a lattice, $f$ is non-constant (as an a.e. defined function). Let $\tilde f : \Gsf \rightarrow \Rb$ be the lift of $f$. Then we can view $\tilde f$ as a bounded non-constant function $\tilde f : \Gsf_2 \times \cdot \times \Gsf_m \rightarrow \Rb$. Further, if $H \subset \Gsf_2 \times \cdot \times \Gsf_m$ is the projection of $\Gamma$, then $\tilde f$ is $H$-invariant. Using Lebesgue differentiation, $\tilde f$ is $\overline{H}$-invariant, where $\overline{H}$ denotes the closure of $H$ in the Hausdorff topology. 

Since $H< \Gsf_2 \times \cdot \times \Gsf_m$ is Zariski dense and normalizes $\overline{H}$, after relabelling we can assume that 
$$
\overline{H} = \prod_{i =2}^k \Gsf_i \times H'
$$
where $H' <  \prod_{i =k+1}^m \Gsf_i$ is discrete. Since $\tilde f$ is non-constant, we must have $k < m$.  

Notice that $H'$ is a Zariski dense discrete subgroup of $\prod_{i =k+1}^m \Gsf_i$. If $H'$ is not a lattice, then the result  follows from induction on the number of simple factors.  So it suffices to assume that $H'$ is a lattice and obtain a contradiction. 

Let $\Gamma_0$ be the kernel of the projection $\Gamma \rightarrow H'$. Then 
$$
\Gamma \backslash \Gsf \rightarrow H' \backslash \prod_{i =k+1}^m \Gsf_i
$$
is a fiber bundle with fibers $\Gamma_0 \backslash \prod_{i =1}^k \Gsf_i$. Further, $f$ descends to a function $f_1 :  H' \backslash \prod_{i =k+1}^m \Gsf_i \rightarrow \Rb$. Hence  
$$
\int_{\Gamma \backslash \Gsf } \abs{f}^2 d\Vol = \Vol\left( \Gamma_0 \backslash \prod_{i =1}^k \Gsf_i\right) \int_{H' \backslash\prod_{i =k+1}^m \Gsf_i } \abs{f_1}^2 d\Vol.
$$
So $\Gamma_0 < \prod_{i =1}^k \Gsf_i$ must be a lattice. Then
$$
+\infty = \Vol(\Ga \ba \Gsf) = \Vol\left( \Gamma_0 \backslash \prod_{i =1}^k \Gsf_i\right) \Vol\left( H' \backslash\prod_{i =k+1}^m \Gsf_i  \right)  < + \infty
$$
 and so we have a contradiction.
\end{proof} 

We now finish the proof of Theorem \ref{thm:singularity discrete group}.
Fix $J \subset \{1,\dots, m\}$, $\Gamma'$, and $\Gsf'$ as in the previous proposition. Let $\mathsf{m}'$ be the push-forward of $\mathsf{m}$ to $\Gsf'$ and let $\nu'$ be the push-forward of $\nu$ to $\Fc':=\Gsf' / \Psf'$ under canonical projections, where $\Psf' : = \prod_{j \in J} \Psf_j$.  Then the support of $\mathsf{m}'$ generates $\Gamma'$ as a group and $\nu'$ is the unique $\mathsf{m}'$-stationary measure. Further, $\nu'$ is non-singular to the Lebesgue measure on $\Fc'$, which is the push-forward of the Lebesgue measure on $\Fc$. 

Hence by replacing  $\Gsf$ by $\Gsf'$ we can assume that for every $1 \leq i \leq m$, there is no $\Gsf_i$-invariant non-zero vector in $L^2(\Gamma \backslash \Gsf)$. By Proposition~\ref{prop:gap result}, 
\begin{equation}\label{eqn:gap with growth fcn} 
\psi_{\Ga} <  \Psi_{\Gsf} \quad \text{on} \quad \fa^+\smallsetminus\{0\}.
\end{equation} 
Now Corollary~\ref{cor:conical has positive measure} and the fact that $\nu$, $\Leb$ are non-singular imply that 
$$
   \Leb(\La^{\rm con}(\Ga)) > 0.  
$$
However this contradicts Proposition~\ref{prop:LO76} and Equation~\eqref{eqn:gap with growth fcn}.
\qed

\section{Proof of Theorem~\ref{thm.general PS}}

Following the argument in \cite[Proposition 7.6]{LO_Dichotomy} (stated as Proposition~\ref{prop:LO76} above), we prove Theorem~\ref{thm.general PS}. As in the previous sections, let $\Gsf$ be a connected semisimple Lie group without compact factors and with finite center.

We start by recalling some facts, established by Quint, about the growth indicator function.

\begin{theorem} \cite[Theorem 8.1]{Quint_PS}\label{thm:ex of PS implies comp to growth indicator}
   Suppose $\Ga < \Gsf$ is a Zariski dense discrete subgroup, $\phi \in \mfa^*$, and $\delta \ge 0$. If a $\delta$-dimensional coarse $\phi$-Patterson--Sullivan measure of $\Ga$ exists, then
   $$
\psi_{\Ga} \le \delta \cdot \phi \quad \text{on} \quad \fa.
$$
\end{theorem}

\begin{remark} Theorem 8.1 in \cite{Quint_PS} only considers non-coarse Patterson--Sullivan measures, but the same proof works in the coarse case. 
\end{remark}

\begin{lemma} \cite[Lemma III.1.3]{Quint_Divergence} \label{lem:div implies comp to growth indicator}
   Suppose $\Ga < \Gsf$ is a discrete subgroup and $\phi \in \mfa^*$ satisfies  $\sum_{\ga \in \Ga} e^{- \phi(\kappa(\ga))} = + \infty$. Then there exists $u \in \fa^+ \smallsetminus \{0\}$ such that
   $$
   \psi_{\Ga}(u) \ge \phi(u).
   $$

\end{lemma}

As our terminology is slightly different, we recall the proof. 

\begin{proof}
   Suppose to the contrary that $\psi_{\Ga} < \phi$ on $\fa^+  \smallsetminus \{0\}$. 
   
   For each unit vector $u \in \fa^+$, choose $\epsilon_u > 0$ such that $\psi_{\Ga}(u) < \phi(u) - \epsilon_u$. Then there exists an open cone $\Cc_u \subset \fa$ containing $u$ such that
   $$
\left| \phi\left( \frac{v}{\|v\|} \right) - \phi(u) \right| < \epsilon_u \text{ for all }v \in \Cc_u \quad \text{and} \quad 
\sum_{\ga \in \Ga, \kappa(\ga) \in \Cc_u} e^{-(\phi(u) - \epsilon_u) \| \kappa(\ga) \|} < + \infty.
   $$
Hence we obtain
   $$
   \sum_{\ga \in \Ga, \kappa(\ga) \in \Cc_u} e^{- \phi(\kappa(\ga))} < + \infty.
   $$
   Since the unit sphere in $\fa^+$ is compact, there exist unit vectors $u_1, \dots, u_n \in \fa^+$ such that $\fa^+ \smallsetminus \{0\} \subset \bigcup_{i = 1}^n \Cc_{u_i}$. We then have
   $$
   \sum_{\ga \in \Ga} e^{- \phi(\kappa(\ga))} < + \infty,
   $$
   which is a contradiction.
\end{proof}

Now we finish the proof of Theorem \ref{thm.general PS}.

\begin{proof}[Proof of Theorem \ref{thm.general PS}] Corollary~\ref{cor:conical has positive measure} and the fact that $\nu$, $\mu$ are non-singular imply that 
$$
   \mu(\La^{\rm con}(\Ga)) > 0.  
$$
By the definition of the conical limit set and the Shadow Lemma (Lemma~\ref{lem.shadow lemma}), we have 
   $$
   \sum_{\ga \in \Ga} e^{-\delta \phi(\kappa(\ga))} = + \infty.
   $$
   Theorem~\ref{thm:ex of PS implies comp to growth indicator} and Lemma~\ref{lem:div implies comp to growth indicator}  imply the remaining assertion. 
\end{proof}

\section{Rank one discrete subgroups with infinite Bowen--Margulis--Sullivan measures}

This section is devoted to the proof of Theorem \ref{thm:rank one case}.
Suppose $X$ is a negatively curved symmetric space with a fixed basepoint $o \in X$ and $\Gsf = \Isom(X)$.

\subsection{Hopf parametrization and Bowem--Margulis--Sullivan measures} \label{subsec:BMS}
Setting $\partial_{\infty}^{2} X := \{ (x, y ) \in \partial_{\infty}X \times \partial_{\infty} X : x \neq y \}$, the Hopf parametrization of the unit tangent bundle $\mathrm{T}^1(X)$ is a map given by
$$
\begin{tikzcd}[row sep=small]
\mathrm{T}^1(X) \arrow[r] & \qquad\quad \partial_{\infty}^{2} X \times \Rb \qquad\quad \\
v \arrow[r, mapsto] &  (v^+, v^-, -B(g^{-1}, v^+))
\end{tikzcd}
$$
where $v^+, v^- \in \partial_{\infty}X$ are the forward and backward endpoints of $v \in \mathrm{T}^1(X) $ under the geodesic flow, and $g \in \Gsf$ is chosen so that $go \in X$ is the basepoint of $v$. One can show that the above map is well-defined, and is indeed a homeomorphism. Moreover, the geodesic flow on $\mathrm{T}^1(X)$ corresponds to the translation on $\Rb$ in the Hopf pararmetrization.

Let $\Ga < \Gsf$ be a discrete subgroup and $\mu$ a $\delta$-dimensional Patterson--Sullivan measure of $\Ga$ on $\partial_{\infty} X$, for $\delta \ge 0$.
Using the Hopf parametrization, we define the Bowen--Margulis--Sullivan measure of $\Ga$ on $\mathrm{T}^1(\Ga \ba X) = \Ga \ba \mathrm{T}^1(X)$, associated to $\mu$.

Consider the Radon measure $\tilde{m}_{\mu}$ on $ \mathrm{T}^1(X) = \partial_{\infty}^2X \times \Rb$ defined by
$$
d \tilde{m}_{\mu}(x, y, t) = e^{2 \delta \langle x, y \rangle_o} d \nu(x) d \nu(y) dt
$$
where $\langle x, y \rangle_o = \lim_{p \to x, q \to y} \frac{1}{2} \left(\dist_X(o, p) + \dist_X(o, q) - \dist_X(p, q) \right)$ is the Gromov product and $dt$ is the Lebesgue measure on $\Rb$. Then the measure $\tilde{m}_{\mu}$ is invariant under the $\Ga$-action and the geodesic flow. Hence, this induces the measure
$$
m_{\mu}^{\BMS} \quad \text{on} \quad \mathrm{T}^1(\Ga \ba X)
$$
which is invariant under the geodesic flow. We call $m_{\mu}^{\BMS}$ the (generalized) \emph{Bowen--Margulis--Sullivan measure} of $\Ga$ associated to $\mu$.

The follwing is a part of the classical Hopf--Tsuji--Sullivan dichotomy for the geodesic flow.

\begin{theorem}[{\cite[Corollary 20, Theorem 21]{Sullivan_density}}, \cite{CI_limitsets}, \cite{roblin}] \label{thm:HTS}
   Let $\Ga < \Gsf$ be a non-elementary discrete subgroup and $\mu$ a $\delta$-dimensional Patterson--Sullivan measure of $\Ga$ on $\partial_{\infty} X$. Then the following are equivalent.
   \begin{itemize}
      \item $\sum_{\ga \in \Ga} e^{-\delta \dist_X(o, \ga o)} = + \infty$.
      \item $\mu(\La^{\rm con}(\Ga)) > 0$.
      \item  $\mu(\La^{\rm con}(\Ga)) = 1$.
      \item The geodesic flow on $(\mathrm{T}^1(\Ga \ba X), m_{\mu}^{\BMS})$ is completely conservative and ergodic.
   \end{itemize}

\end{theorem}

\subsection{Proof of Theorem \ref{thm:rank one case}}
    We fix $\Ga$, $\mathsf{m}$, and $\mu$ as in the statement:
    \begin{itemize}
      \item $\Ga < \Gsf$ is a non-elementary discrete subgroup.
      \item $\mathsf{m}$ is a probability measure on $\Ga$ such that
      $$\langle \supp \mathsf{m} \rangle = \Ga \quad \text{and} \quad \sum_{\ga \in \Ga} \dist_X(o, \ga o) \mathsf{m}(\ga) < + \infty.$$
      \item $\mu$ is a Patterson--Sullivan measure of $\Ga$ on $\partial_{\infty} X$.
    \end{itemize} Suppose to the contrary that the $\mathsf{m}$-stationary measure $\nu$ and $\mu$  are non-singular. 

    Note that in this rank one setting, the proof of Corollary \ref{cor:conical has positive measure} works for a non-elementary discrete subgroup $\Ga < \Gsf$, which may not be Zariski dense. Hence, by Corollary \ref{cor:conical has positive measure}, we have $\mu(\La^{\rm con}(\Ga)) > 0$. By Theorem \ref{thm:HTS}, the geodesic flow on $\mathrm{T}^{1}(\Ga \ba X)$ is completely conservative and ergodic with respect to $m_{\mu}^{\BMS}$. 

For $x \in \partial_{\infty} X$, let $\sigma_x : \Rb \rightarrow X$ be the unit speed geodesic line such that $\sigma_x(0) = o$ and $\lim_{t \to + \infty} \sigma_x(t) = x$.   Since $m_{\mu}^{\BMS}(\mathrm{T}^{1}(\Ga \ba X)) = + \infty$, it follows from the Hopf ratio ergodic theorem that for any $R > 0$ we have 
$$
\lim_{t \rightarrow \infty} \frac{1}{T} \int_0^T \mathbf{1}_{\Gamma \cdot B_R(o)}( \sigma_x(t)) dt = 0
$$ 
for $\mu$-a.e. $x \in \partial_\infty X$, where $B_R(o) \subset X$ is the closed ball of radius $R > 0$ with center $o \in X$. Hence the following proposition yields a contradiction and Theorem~\ref{thm:rank one case} follows.

\begin{proposition} If $R > 0$ is sufficiently large, then 
$$
\liminf_{T \rightarrow +\infty} \frac{1}{T} \int_0^T \mathbf{1}_{\Gamma \cdot B_R(o)}( \sigma_x(t)) dt > 0
$$ 
for $\nu$-a.e. $x \in \partial_\infty X$.
\end{proposition} 

\begin{proof} 
By \cite[Remark following Theorem 7.7]{Kaimanovich2000}, for $\mathsf{m}^{\Nb}$-a.e. $\mathbf{g} = (g_1,g_2, \dots) \in \Gsf^{\Nb}$ the limit 
$$
\zeta(\mathbf{g}) := \lim_{n \rightarrow \infty} g_1 \cdots g_n o
$$
exists in $\partial_\infty X$ and 
$$
\nu = \zeta_* \mathsf{m}^{\Nb}. 
$$
Note that in Section~\ref{subsec:RW}, $\zeta(\mathbf{g})$ was denoted by $x_{\omega}$ where $\omega = (\omega_n) \in \Gsf^{\Nb}$, $\omega_n = g_1 \cdots g_n$. In this proof, we consider the sequence $\mathbf{g} \in \Gsf^{\Nb}$ for steps instead.

By Thoerem~\ref{thm:conical RW}, we can fix $R_0 > 0$ such that for $\mathsf{m}^{\Nb}$-a.e. $\mathbf{g} = (g_n) \in \Gsf^{\Nb}$, 
   $$
\liminf_{N \to \infty} \frac{1}{N} \# \{ 1 \le n \le N : \dist_X(g_1\cdots g_n o , \sigma_{\zeta(\mathbf{g})}) \leq R_0 \} > 1/2.
   $$
For each $k \in \Nb$, let
$$
A_k : = \{ \mathbf{g}  \in \Gsf^{\Nb} : \dist_X( o, g_1\cdots g_n o) > 4R_0 \text{ for all } n \geq k\}. 
$$
Since 
$$
\lim_{n \rightarrow \infty} \dist_X ( o, g_1\cdots g_no) = +\infty
$$
for $\mathsf{m}^{\Nb}$-a.e.  $\mathbf{g} = (g_n) \in \Gsf^{\Nb}$, we have
$$
\lim_{k \rightarrow \infty} \mathsf{m}^{\Nb}( A_k) = 1.
$$
Hence we can fix $k \in \Nb$ such that $\mathsf{m}^{\Nb}( A_k) > 1/2$. Recall that the shift map $S : (\Gsf^{\Nb}, \mathsf{m}^{\Nb}) \rightarrow (\Gsf^{\Nb}, \mathsf{m}^{\Nb})$ given by 
$$
S(g_1,g_2,g_3, \dots) = (g_2,g_3,\dots)
$$
is ergodic. Hence for $\mathsf{m}^{\Nb}$-a.e. $\mathbf{g} \in \Gsf^{\Nb}$, 
$$
\lim_{N \rightarrow \infty} \frac{1}{N} \# \left\{1 \le  n \le N: S^{n} \mathbf{g} \in A_k \right\}= \lim_{N \rightarrow \infty} \frac{1}{N}\sum_{n=1}^{N} \mathbf{1}_{A_k}( S^n \mathbf{g}) =  \mathsf{m}^{\Nb}( A_k) > 1/2.
$$
Thus there exists $c > 0$ such that for $\mathsf{m}^{\Nb}$-a.e. $\mathbf{g} \in \Gsf^{\Nb}$, we have  
\begin{equation}\label{eqn:defn of c}
c \le \liminf_{N \to \infty} \frac{1}{N} \# \left\{1 \le  n \le N: S^{n} \mathbf{g} \in A_k \text{ and } \dist_X(g_1\cdots g_n o , \sigma_{\zeta(\mathbf{g})}) \leq R_0 \right\}.
\end{equation} 

Since $\mathsf{m}$ has a finite first moment, by \cite[Theorem 1.2]{MaherTiozzo2018}, there exists $\ell(\mathsf{m}) > 0$ such that for $\mathsf{m}^{\Nb}$-a.e. $\mathbf{g} \in \Gsf^{\Nb}$, 
\begin{equation} \label{eqn:linear growth}
\lim_{n \rightarrow \infty} \frac{1}{n} \dist_X(o, g_1 \cdots g_n o) = \ell(\mathsf{m}).
\end{equation}

Since $\nu = \zeta_* \mathsf{m}^{\Nb}$, for $\nu$-a.e. $x \in \partial_\infty X$ there exists $\mathbf{g} \in \Gsf^{\Nb}$ such that $\zeta(\mathbf{g}) =x$ and  Equations~\eqref{eqn:defn of c} and \eqref{eqn:linear growth} hold. Since $\supp \mathsf{m}$ generates $\Ga$, we may assume that $\mathbf{g} \in \Ga^{\Nb}$.
So it suffices to fix  such $x \in \partial_{\infty} X$ and $\mathbf{g} \in \zeta^{-1}(x)$, and then show that 
$$
\liminf_{T \rightarrow + \infty} \frac{1}{T} \int_0^T \mathbf{1}_{\Gamma \cdot B_{2R_0}(o)}( \sigma_x(t)) dt > 0.
$$ 

 Let 
$$
I_0 := \left\{ n \in \Nb  : S^{n} \mathbf{g} \in A_k \text{ and } \dist_X(g_1\cdots g_n o , \sigma_{x}) \leq R_0 \right\}
$$
and let $I := \{ n_1 < n_2 < \cdots\} \subset I_0$ be a maximal $k$-separated set, i.e., $I_0$ is a maximal set such that $|n_i - n_j| \ge k$ for all distinct $n_i, n_j \in I_0$. By maximality, 
$$
I_0 \subset \bigcup_{ i \in I} (i-k, i+k]
$$
and hence for all $j \in \Nb$,
$$
\# (I_0 \cap [0,n_j]) \leq 2k \cdot \# (I \cap [0,n_j]).
$$
 Thus by Equation~\eqref{eqn:defn of c},
\begin{equation} \label{eqn:distribution of I}
\frac{c}{2k} \leq \liminf_{j \rightarrow \infty} \frac{1}{n_j} \# (I \cap [0,n_j]).  
\end{equation}

For each $j \in \Nb$, fix $t_j > 0$ such that  
$$
\dist_X(g_1 \cdots g_{n_j} o, \sigma_{x}(t_j)) \leq R_0.
$$
 Then $t_j \rightarrow +\infty$ and hence we can fix a subsequence $t_{j_i}$ such that 
 $$
 t_j < t_{j_i}
 $$
 for all $j < j_i$. 
 
Notice that if $j < j'$, then 
\begin{align*}
\abs{t_j - t_{j'}} & = \dist_X(\sigma_{x}(t_j), \sigma_{x}(t_{j'})) \geq -2R_0 + \dist_X(g_1 \cdots g_{n_j} o, g_1 \cdots g_{n_{j'}} o) \\
& =  -2R_0 + \dist_X(o, g_{n_j + 1} \cdots g_{n_{j'}} o) > 2R_0
\end{align*} 
since $S^{n_j}(\mathbf{g}) = (g_{n_j + 1}, g_{n_j+2}, \cdots) \in A_k$ and $n_{j'} - n_j \ge k$. Thus
$$
[t_j - R_0, t_j+R_0] \cap [t_{j'} - R_0, t_{j'}+R_0] = \emptyset.
$$
Then since 
$$
\dist_X(g_1 \cdots g_{n_j} o, \sigma_{x}(t)) \leq 2R_0
$$
when $t \in [t_j - R_0, t_j+R_0] $, we have for all $i \in \Nb$ that
\begin{align*}
\int_{-R_0}^{t_{j_i}+R_0} \mathbf{1}_{ \Gamma \cdot B_{2R_0}(o)}( \sigma_{x}(t)) dt \geq  \sum_{j=1}^{j_i} \int_{t_j-R_0}^{t_j+R_0} dt =2R_0 \cdot j_i = 2R_0 \cdot \# (I \cap [0,n_{j_i}]).
\end{align*} 
Since 
$$
\ell(\mathsf{m}) = \lim_{i \rightarrow \infty} \frac{1}{n_{j_i}} \dist_X(o, g_1 \cdots g_{n_{j_i}} o) = \lim_{i \rightarrow \infty} \frac{1}{n_{j_i}} t_{j_i},
$$
then by Equation~\eqref{eqn:distribution of I},
$$
\begin{aligned}
\liminf_{i \rightarrow \infty} \frac{1}{t_{j_i}+R_0}\int_0^{t_{j_i}+R_0} \mathbf{1}_{ \Gamma \cdot B_{2R_0}(o)}( \sigma_{x}(t)) dt & \geq  \liminf_{i \rightarrow \infty} \frac{2R_0}{\ell(\mathsf{m}) n_{j_i}}  \# (I \cap [0,n_{j_i}]) \\
&  \geq \frac{cR_0}{ \ell(\mathsf{m}) k}.
\end{aligned} 
$$
Therefore, the proposition holds for any $R > 2R_0$.
\end{proof}

\bibliographystyle{alpha}
\bibliography{geom}

\end{document}